\documentclass[reqno,12pt]{amsart}
\usepackage{amssymb,amsbsy}
\usepackage[normalem]{ulem}
\usepackage{verbatim}

\theoremstyle{plain}
\newtheorem{lemma}{Lemma}[section]
\newtheorem{theorem}[lemma]{Theorem}
\newtheorem{proposition}[lemma]{Proposition}
\newtheorem{corollary}[lemma]{Corollary}
\newtheorem{example}[lemma]{Example}

\theoremstyle{remark}
\newtheorem{remark}{Remark}
\newtheorem{notation}{Notation}
\newtheorem{organization}{Organization of the paper}

\numberwithin{equation}{section}

\newcommand{\seq}{\subseteq}
\newcommand{\snqq}{\subsetneqq}
\newcommand{\stm}{\setminus}

\newcommand{\longc}{,\dotsc,}
\newcommand{\longe}{=\dotsb=}
\newcommand{\longp}{+\dotsb+}

\newcommand{\longge}{\ge\dotsb\ge}
\newcommand{\longu}{\cup\dotsb\cup}
\newcommand{\longd}{\mid\dots\mid}
\newcommand{\longs}{\seq\dotsb\seq}
\newcommand{\longop}{\oplus\dotsb\oplus}

\newcommand{\mmod}[1]{\!\!\pmod{#1}}

\newcommand{\<}{\langle}
\renewcommand{\>}{\rangle}
\newcommand{\lpr}{\left(}
\newcommand{\rpr}{\right)}
\newcommand{\lfr}{\left\{}
\newcommand{\rfr}{\right\}}
\newcommand{\lfl}{\left\lfloor}
\newcommand{\rfl}{\right\rfloor}

\DeclareMathOperator{\diam}{diam}
\DeclareMathOperator{\lcm}{lcm}
\DeclareMathOperator{\ord}{ord}
\DeclareMathOperator{\rk}{rk}

\newcommand{\alp}{\alpha}
\newcommand{\bet}{\beta}
\newcommand{\eps}{\varepsilon}
\newcommand{\lam}{\lambda}
\renewcommand{\phi}{\varphi}

\newcommand{\oA}{\overline A}

\newcommand{\N}{{\mathbb N}}
\newcommand{\Z}{{\mathbb Z}}
\newcommand{\R}{{\mathbb R}}

\newcommand{\s}{{\mathbf s}}
\renewcommand{\t}{{\mathbf t}}

\begin{document}
\baselineskip=16pt

\title[Generating an abelian group]{How long does it take to generate
  a group?}

\author{Benjamin Klopsch}
\address{Mathematisches Institut \\ Heinrich-Heine-Universit\"at \\
  40225 D\"usseldorf \\ Germany}
\email{klopsch@math.uni-duesseldorf.de}

\author{Vsevolod F. Lev}
\address{Department of Mathematics \\ Haifa University at Oranim \\
  Tivon 36006 \\ Israel}
\email{seva@math.haifa.ac.il}

\thanks{At the time our collaboration started the first author held a
  PIMS postdoctoral fellowship at the University of Alberta (Canada)
  while the second author was a postdoctoral fellow at the Hebrew
  University of Jerusalem (Israel). We express our gratitude for the
  support provided by these institutions.}

\begin{abstract}
The diameter of a finite group $G$ with respect to a generating set $A$ is
the smallest non-negative integer $n$ such that every element of $G$ can be
written as a product of at most $n$ elements of $A \cup A^{-1}$. We denote
this invariant by $\diam_A(G)$. It can be interpreted as the diameter of the
Cayley graph induced by $A$ on $G$ and arises, for instance, in the context
of efficient communication networks.

In this paper we study the diameters of a finite \emph{abelian} group $G$
with respect to its various generating sets $A$. We determine the maximum
possible value of $\diam_A(G)$ and classify all generating sets for which
this maximum value is attained. Also, we determine the maximum possible
cardinality of $A$ subject to the condition that $\diam_A(G)$ is ``not too
small''. Connections with caps, sum-free sets, and quasi-perfect codes are
discussed.
\end{abstract}

\maketitle


\section{Introduction}\label{s:intro}

Let $G$ be a finite group and let $A$ be a subset of $G$. The
\emph{subgroup generated by $A$ in $G$} is
$$
\<A\> := \bigcap\,\{H \leq G \colon A \seq H\},
$$
the intersection of all subgroups containing $A$. This is the
smallest subgroup of $G$ lying above $A$. If $\<A\> = G$ then $A$ is
called a \emph{generating set for $G$}; so $A$ generates $G$ if and
only if it is not contained in any proper subgroup of $G$. Loosely
speaking, the aim of this paper is to determine whether and how
quickly $A$ generates $G$, using only information about the
cardinality of $A$.

For this we regard generation as a dynamic step-by-step process. Since
eventually we restrict to the situation where $G$ is abelian, we use
additive notation throughout. Let $A^\pm := (-A) \cup \{0\} \cup A$
denote the ``symmetric closure'' of $A$ with respect to addition. Thus
$\<A\>$ consists of all those $g \in G$ representable as a sum of
elements of $A^\pm$, and for every non-negative integer $\rho$ we
define
$$
\<A\>_\rho := \rho A^\pm = \{a_1 \longp a_\rho \colon a_i \in
A^\pm\} \seq G,
$$
the set of all $g\in G$ representable as a sum of \emph{at most
  $\rho$} elements of $A^\pm$. Evidently, these sets form an ascending
chain $\{0\}=\<A\>_0\seq\<A\>_1\seq\dots$, and their union
$\bigcup_{\rho\ge 0}\<A\>_\rho=\<A\>$ provides a bottom-up description
of the subgroup generated by $A$. We notice that
$\<A^\pm\>_\rho=\<A\>_\rho$ for every $\rho \in \N_0:=\N\cup\{0\}$.

We define the \emph{diameter of $G$ with respect to $A$} as
$$
\diam_A(G) := \min \{ \rho \in \N_0 \colon \< A\>_\rho = G \}.
$$
To explain this nomenclature we introduce a suitable length function. For
every $g\in G$ we define the \emph{length of $g$ with respect to $A$} as
$$
l_A(g) := \min \{ \rho \in \N_0 \colon g \in \<A\>_\rho \}.
$$
Here we agree that $\min \varnothing = \infty$, so that for $g\in
G$ we have $g\in\<A\>$ if and only if $l_A(g)<\infty$. Observe that
$l_A(g)$ is simply the minimum number of (not necessarily distinct)
elements of $A^\pm$ required to represent $g$ as their sum. Thus we
have
\begin{align*}
  \< A \>_\rho & = \{ g \in G \colon l_A(g) \leq \rho \} \\
\intertext{and}
  \diam_A(G) & = \max \{ l_A(g) \colon g \in G\}.
\end{align*}
{}From this point of view the diameter allows a simple graph-theoretic
interpretation. Indeed, $l_A(g)$ is the distance from zero to $g$ in the
Cayley graph induced by $A$ on $G$, and $\diam_A(G)$ is the diameter of this
graph. Moreover, $\diam_A(G)$ describes quantitatively ``how long'' it takes
to generate $G$ from $A$. In particular, $A$ is a generating set for $G$ if
and only if $\diam_A(G) < \infty$.

The length of a given element and the diameter of $G$ depend upon the choice
of $A \seq G$. In contrast, the \emph{absolute diameter}
$$
\diam(G) := \max \{ \diam_A(G) \colon \<A\> = G \}
$$
is an invariant of $G$ itself: every generating set produces the group in at
most that many steps.

The two problems addressed in this paper are:
\begin{itemize}
\item[\textbf{P1.}] What is the value of $\diam(G)$ and what are the
  generating sets $A\seq G$ such that $\diam_A(G) = \diam(G)$?
\item[\textbf{P2.}] How large can $\diam_A(G)$ be, given that $|A|$ is
  large?
\end{itemize}

Under the assumption that $G$ is \emph{abelian}, we solve the first
problem completely (see Theorems~\ref{t:thm21} and \ref{t:thm22}) and
we provide several partial answers to the second one (see
Theorems~\ref{t:thm27}, \ref{t:thm29}, and \ref{t:thm212}). Our
results are presented in detail in the next section.

At first sight the restriction to abelian groups may seem too strong ---
after all, finite abelian groups are rather trivial objects from a
group-theoretic point of view. But here we are concerned with combinatorial
properties of finite groups, and already the case where $G$ is homocyclic
leads to interesting applications in coding theory and other areas; see
Section~\ref{s:context}.

\begin{remark}
  Our starting point in this paper are representations of the elements
  of $G$ by ``algebraic sums'' $\pm a_1\pm\dotsb\pm a_n$. However, it
  would be equally natural to allow only ``pure sums'' $a_1\longp
  a_n$. Formally, we could have put $A^+ := A\cup\{0\}$ and then
  proceeded to define
  $\<A\>_\rho^+,\,\diam_A^+(G),\,l_A^+(g),\,\diam^+(G)$, etc. This
  approach leads to a significantly different theory, which we are
  going to cover in a separate paper.
\end{remark}

\begin{organization}
  In Sections~\ref{s:results} and \ref{s:context} we describe our
  results in detail and place them into a broader context.
  Section~\ref{s:auxiliary} contains a list of auxiliary results. All
  proofs are collected in Sections~\ref{s:abs_diam} to
  \ref{s:bounds_IV}.
\end{organization}

\begin{notation}
  The notation used is mostly standard, but the following list may be
  of help.
  \begin{tabbing}
    {gcd(n 1, . . .,n r) } \= \kill
    \, $\N,\N_0$, and $\Z$ \> the sets of positive, non-negative, and
    all integers, respectively.
  \end{tabbing}
For natural numbers $n,n_1, \ldots, n_r$:
  \begin{tabbing}
    {gcd(n 1, . . .,n r) } \= \kill
    \, $\Z_n$ (or $\Z / n\Z$) \> the group of residues modulo $n$; \\
    \, $\gcd(n_1\longc n_r)$ \> the greatest common divisor of $n_1\longc
     n_r$; \\
    \, $\lcm(n_1\longc n_r)$ \> the least common multiple of $n_1 \longc
     n_r$.
  \end{tabbing}
  For real numbers $x,y$:
  \begin{tabbing}
    {gcd(n 1, . . .,n r) } \= \kill
    \, $\lfloor x\rfloor$ \> the greatest integer not exceeding $x$; \\
    \, $\lceil x\rceil$ \> the smallest integer no smaller than $x$; \\
    \, $[x,y]$ \> the set of all integers $n$ satisfying $x \le n
     \le y$.
  \end{tabbing}
  For a non-negative integer $r$ and subsets $A,A_1 \longc A_r$ of an
  abelian group $G$:
  \begin{tabbing}
    {gcd(n 1, . . .,n r) } \= \kill
    \, $A_1 \longp A_r$ \> the set $\{a_1 \longp a_r \colon a_1 \in A_1
    \longc a_r \in A_r \}$; \\
    \, $rA$ \> the set $A\longp A$ ($r$ summands); \\
    \, $r \ast A$ \> the set $\{ra \colon a\in A\}$.
  \end{tabbing}
  (For $r=0$ the expressions $A_1\longp A_r$ and $rA$ should be
  interpreted as $\{0\}$.)

  Also, for integers $m$ and $n$ we write $m \mid n$ to indicate that
  $m$ divides $n$.
\end{notation}


\section{Summary of results}\label{s:results}

We state our results in several blocks: first we look at the absolute
diameter of a finite abelian group and then at the maximum size of
small diameter sets; cf.\ Problems~P1 and~P2 of the Introduction.

\subsection{The absolute diameter}

Up to isomorphism, every finite abelian group is completely
characterized by its type. Let $G$ be a finite abelian group of type
$(m_1,\ldots,m_r)$; that is, $G\cong\Z_{m_1}\longop\Z_{m_r}$ where
$1\neq m_1\longd m_r$. The number $r$ is called the \emph{rank of $G$}
and denoted $\rk(G)$; it is the minimum number of elements required to
generate $G$. A \emph{standard generating set for $G$} is a subset
$A=\{a_1\longc a_r\}\seq G$ such that $G=\<a_1\>\longop\<a_r\>$ and
$\ord(a_i)=m_i$ for all $i \in [1,r]$. It is worth pointing out that
the trivial group $\{ 0 \}$ has type $()$, rank $0$, and precisely one
standard generating set, namely $\varnothing$.

Perhaps not surprisingly, the diameter of $G$ with respect to a
standard generating set is as large as possible. Indeed, if $A$ is a
standard generating set for $G$, then plainly $\diam_A(G) =
\sum_{i=1}^r \lfloor m_i/2 \rfloor$, and we establish

\begin{theorem}\label{t:thm21}
A finite abelian group $G$ of type $(m_1,\ldots,m_r)$ has diameter
$$
\diam(G) = \sum_{i=1}^r \lfloor m_i/2 \rfloor.
$$
\end{theorem}

Conversely, if $A$ is a generating set for $G$ with
$\diam_A(G)=\diam(G)$, then $A$ is ``nearly standard''; a small
wrinkle is observed if $G$ has invariant factors of order three. To
make this precise, for $m\in\N$ we define
\begin{equation}\label{def_of_nu}
  \nu_m(G) := \#\{i\in[1,r]\colon m_i=m\},
\end{equation}
the number of components of $(m_1\longc m_r)$ equal to $m$. We
notice that, if $m$ is prime and $\nu_m(G) \ne 0$, then $m_i=m$ for
all $i\in[1,\nu_m(G)]$. Furthermore, if $m'$ and $m''$ are co-prime,
then at least one of $\nu_{m'}(G)$ and $\nu_{m''}(G)$ is zero.

\begin{theorem}\label{t:thm22}
  Let $G$ be a finite abelian group of rank $r$. Then for every subset
  $A \seq G$ the following assertions are equivalent:
  \begin{itemize}
  \item[(i)] $\diam_A(G) = \diam(G)$;
   \item[(ii)] there exists a standard generating set $B=\{b_1\longc
    b_r\}$ for $G$ such that
    $$
    B^\pm \seq A^\pm \seq \big(B \cup \{ b_{2i-1}+b_{2i} \colon i
    \in [1,\nu_3(G)/2] \}\big)^\pm.
    $$
  \end{itemize}
\end{theorem}

\begin{corollary}\label{c:cor23}
  Let $G$ be a finite abelian group. Then for every subset
    $A\seq G$ satisfying $\diam_A(G) = \diam(G)$ we have
  $$
  \rk(G) \le |A| \le 1 + 2\rk(G) - \nu_2(G) + 2\lfloor \nu_3(G)/2
  \rfloor.
  $$
  Indeed, both bounds are sharp.
\end{corollary}

For abelian groups these theorems provide a complete solution to
Problem~P1 of the Introduction. In contrast, for no infinite family
$\mathcal{H}$ of non-abelian finite groups such explicit formulae for
the diameters $\diam(H)$, $H \in \mathcal{H}$, seem to be known.


\subsection{The maximum size of small diameter sets:
  definitions}\label{s:subsec22}

Let $G$ be a finite abelian group. Intuitively, it is clear that the
larger a subset $A \seq G$, the smaller the corresponding diameter
$\diam_A(G)$. But to what extent does the size of $A$ alone guarantee
fast generation?

Suppose that $\rho \in \N$. We want to find an upper bound for the
sizes of generating sets $A$ for $G$ with $\diam_A(G) \ge \rho$.
Agreeing that $\max\varnothing=0$, we define
$$
\s_\rho(G) := \max \{ |A| \colon A \seq G \text{ such that }
\rho\le\diam_A(G) < \infty \}.
$$
Equivalently, we have $\s_\rho(G) = \max\{ |A|\colon A\seq G \text{
  such that } \<A\>_{\rho-1} \neq \<A\>=G \}$. The significance of
this new invariant stems from the observation that every generating
set $A$ for $G$ of size larger than $\s_\rho(G)$ surely generates $G$
in less than $\rho$ steps.

In order to better understand $\s_\rho(G)$ we introduce a related and perhaps
more fundamental invariant, $\t_\rho(G)$; this requires a short preparation.
Recall that the \emph{period} of a subset $S\seq G$ is the subgroup
$\pi(S):=\{g\in G\colon S+g=S\}\le G$, and $S$ is \emph{periodic} or
\emph{aperiodic} according to whether $\pi(S)\neq\{0\}$ or $\pi(S)=\{0\}$.
Clearly, $S$ is a union of $\pi(S)$-cosets; consequently, if
$\gcd(|S|,|G|)=1$ then $S$ is aperiodic. Another useful observation is that
the image of $S$ under the canonical homomorphism $G\to G/\pi(S)$ is an
aperiodic subset of the quotient group $G/\pi(S)$.

We say that a subset $A\seq G$ is \emph{$\rho$-maximal} if it is maximal
(under inclusion) subject to $\diam_A(G)\ge\rho$; that is, subject to
$\<A\>_{\rho-1}\neq G$. Plainly, we have
$$
\s_\rho(G) = \max \{ |A| \colon \text{$A$ is a $\rho$-maximal
  generating set for $G$}\}.
$$

One can construct $\rho$-maximal generating sets by a ``lifting process'' as
follows. Suppose that $H \lneqq G$, and let $\phi: G \rightarrow G/H$ denote
the canonical homomorphism. If $\oA = \oA{\;}^\pm$ is a generating set for
$G/H$, then the full pre-image $A := \phi^{-1}(\oA)$ is a generating set for
$G$ with $\diam_A(G) = \diam_{\oA}(G/H)$ and $\pi(A)\ge H$. Conversely,
suppose $A = A^\pm$ is a generating set for $G$ and $H \leq \pi(A)$. Then the
image $\oA = \phi(A)$ is a generating set for $G/H$ and $\diam_{\oA}(G/H) =
\diam_A(G)$. Furthermore, if $A$ is $\rho$-maximal then so is $\oA$, and if
$H = \pi(A)$ then $\oA$ is aperiodic. This shows that every $\rho$-maximal
generating set is induced by an aperiodic one. A natural point of view,
therefore, is to consider \emph{aperiodic} $\rho$-maximal generating sets as
``primitive'' and concentrate on their properties first. Accordingly, we let
$$
\t_\rho(G) := \max \{ |A| \colon \text{$A$ is an aperiodic
  $\rho$-maximal generating set for $G$} \},
$$
again subject to the agreement that $\max \varnothing = 0$.

We notice that, if $\rho > \diam(G)$, then $\rho > \diam_A(G)$ for all
generating subsets $A \subseteq G$; consequently, $G$ admits no
$\rho$-maximal generating sets and $\t_\rho(G) = 0$. On the other hand,
if $\rho \in [1,\diam(G)]$ then in the definition of $\t_\rho(G)$ we can
safely disregard the requirement that $A$ generates $G$. Indeed, suppose
that $A \seq G$ is $\rho$-maximal, but does not generate $G$. Then $A$ is
a proper subgroup of $G$. Suppose that in addition $A$ is aperiodic. Then
$A = \{0\}$ lies in the symmetric closure $B^\pm$ of every subset $B
\subseteq G$; it follows that $\diam(G) < \rho$. In summary, we have
$$
\t_\rho(G) =
\begin{cases}
  \max \{ |A| \colon \text{$A \seq G$ is aperiodic and $\rho$-maximal}
  \} & \text{if $\rho \in [1,\diam(G)]$,} \\
  0 & \text{if $\rho > \diam(G)$.}
\end{cases}
$$

The close connection between $\t_\rho(G)$ and $\s_\rho(G)$ is
described by

\begin{lemma}\label{l:lem24}
  Let $G$ be a finite abelian group and let $\rho \in [2,\diam(G)]$.
  Then
  $$
  \s_\rho(G) = \max \{ |H| \cdot \t_\rho(G/H) \colon H \lneqq G \}.
  $$
\end{lemma}

Unless $G$ is trivial, the only $1$-maximal subset of $G$ is $G$
itself; so $\s_1(G) = |G|$ and $\t_1(G)=0$. In fact, it can happen
that $\t_\rho(G) = 0$ also for $\rho \in [2,\diam(G)]$.  Examples of
this kind can be constructed as follows.

\begin{example}\label{p:exam25}
  Let $n \in \N$ with $n \geq 2$, and let $G = \Z_2 \oplus
  \Z_{2^{n+1}}$. Then $\diam(G) = 1 + 2^n$, and every $2^n$-maximal
  subset of $G$ is periodic.
\end{example}

Perhaps it is feasible to classify all pairs $(G,\rho)$ such that
$\t_\rho(G) = 0$, but beyond Example~\ref{p:exam25} nothing much is
known.


\subsection{The maximum size of small diameter sets: explicit formulae}

We completely determine $\t_\rho(G)$ and $\s_\rho(G)$ in the following
particular cases:
\begin{itemize}
\item[(i)] $G$ is an arbitrary finite abelian group and
  $\rho\in\{1,2,3,\diam(G)\}$;
\item[(ii)] $G$ is a finite cyclic group and $\rho\in [1,\diam(G)]$ is
  arbitrary.
\end{itemize}
The case $\rho=1$ is settled by the observation preceding
Example~\ref{p:exam25}. The case $\rho=2$ is not difficult either.

\begin{proposition}\label{p:prop26}
  Let $G$ be a finite abelian group with $\diam(G) \geq 2$. Then
  $$
  \t_2(G) = \s_2(G) =
  \begin{cases}
    |G|-1 & \text{if $|G|$ is even,} \\
    |G|-2 & \text{if $|G|$ is odd.}
  \end{cases}
  $$
\end{proposition}

For the next theorem recall that the \emph{exponent} of an abelian
group $G$ is $\exp(G) = \max\{\ord(g) \colon g\in G\}$, where
$\ord(g)$ denotes the order of $g \in G$. The \emph{$2$-rank} of $G$,
denoted $\rk_2(G)$, is the rank of the Sylow $2$-subgroup of $G$. The
group $G$ is called \emph{homocyclic} if it can be written as a direct
sum of pairwise isomorphic cyclic groups.

Suppose that $G$ is of type $(m_1\longc m_r)$. Then we have $\rk(G)=r$ and,
unless $G$ is trivial, $\exp(G)=m_r$. Moreover, $\rk_2(G)$ is the number of
indices $i\in[1,r]$ for which $m_i$ is even, and $\rk_2(2\ast G)$ is the
number of indices $i\in[1,r]$ for which $m_i$ is divisible by four. (Here $2
\ast G$ denotes the set $\{ 2g \colon g \in G \}$; see our list of
definitions at the end of Section~\ref{s:intro}.) Clearly, $G$ is homocyclic
if and only if $m_1 \longe m_r$.

\begin{theorem}\label{t:thm27}
  Let $G$ be a finite abelian group with $\diam(G)\ge 3$.
\begin{itemize}
\item[(i)] Suppose that $|G|$ is odd, and let $n:=\exp(G)$. Then
  $$ \t_3(G) = \s_3(G) =
    \begin{cases}
      \frac{n-1}{2n}\,|G| - 1 & \text{if $n \equiv 1\mmod 4$,} \\
      \frac{n-1}{2n}\,|G|     & \text{if $n \equiv 3\mmod 4$.}
    \end{cases}
  $$
\item[(ii)] Suppose that $|G|$ is even. Then
\begin{align*}
  \t_3(G) &= \begin{cases}
      (|G|-\sqrt{|G|})/2 & \text{if $G$ is homocyclic of exponent $4$,} \\
      |G|/2 -1 & \text{if $\rk_2(G) = \rk_2(2\ast G)$ and $\exp(G)>4$,} \\
      |G|/2    & \text{if $\rk_2(G) > \rk_2(2 \ast G)$,}
             \end{cases} \\
  \s_3(G) &= \begin{cases}
      |G|/2 -1 & \text{if $G$ is a cyclic $2$-group,} \\
      |G|/2    & \text{otherwise.}
             \end{cases}
\end{align*}
\end{itemize}
\end{theorem}

The proof of this theorem already requires a fair amount of work.
Probably, it will be significantly more difficult to find explicit
formulae for $\t_4(G)$ and $\s_4(G)$. Some general estimates are given
later in this section.

Our next result is a consequence of Theorem~\ref{t:thm22}; see also
Corollary~\ref{c:cor23}.

\begin{corollary}\label{c:cor28}
  Let $G$ be a finite abelian group, and suppose that $\rho :=
  \diam(G)\ge 2$. Then
  $$
  \t_\rho(G) = \s_\rho(G) = 1 + 2\rk(G) - \nu_2(G) + 2 \lfloor
  \nu_3(G)/2 \rfloor.
  $$
\end{corollary}

We now turn our attention towards cyclic groups. Notice that
$\diam(\Z_m)=\lfl m/2 \rfl$ for all $m\in\N$, by Theorem~\ref{t:thm21}.

\begin{theorem}\label{t:thm29}
  Let $m \in \N$ and $\rho \in [2, m/2]$.  Then
  $$
  \t_\rho(\Z_m) = 2 \lfl \frac{m-2}{2(\rho-1)} \rfl +1
  $$
  and
  $$
  \s_\rho(\Z_m) = \max \lfr \frac md \lpr 2 \lfl
  \frac{d-2}{2(\rho-1)} \rfl +1 \rpr \colon d\mid m,\ d\ge 2 \rho
  \rfr.
  $$
\end{theorem}

In particular, for cyclic groups of prime order this reduces to

\begin{corollary}\label{c:cor210}
Let $p\ge 5$ be a prime and suppose that $\rho \in [2,(p-1)/2]$. Then
$$
  \t_\rho(\Z_p) = \s_\rho(\Z_p) = 2 \lfl \frac{p-2}{2(\rho-1)} \rfl + 1.
$$
\end{corollary}


\subsection{The maximum size of small diameter sets: estimates}

Let $G$ be a finite abelian group and let $\rho\in [1,\diam(G)]$. For
$\rho\ge 4$ the exact values of $\t_\rho(G)$ and $\s_\rho(G)$ are
likely to depend --- in an increasingly complicated way --- on the
algebraic structure of $G$. Nevertheless, some general estimates in
terms of $\rho$ and $|G|$ can be given. To some extent such estimates
compensate for the lack of explicit formulae as those we were able to
provide for $\rho\le 3$. Moreover, the bounds given by the following
proposition are used in the proofs of several of the results stated
above.

\begin{proposition}\label{p:prop211}
  Let $G$ be a finite abelian group, and suppose that $\rho\ge 2$.
  Then
  $$
  \t_\rho(G)\le\lfl \frac{|G|-2}{\rho-1} \rfl +1.
  $$
  Moreover, if $\rk_2(G)\le 1$ then
  $$
  \t_\rho(G)\le 2 \lfl \frac{|G|-2}{2(\rho-1)} \rfl +1.
  $$
\end{proposition}

Our two final results provide further bounds.

\begin{theorem}\label{t:thm212}
Let $G$ be a finite abelian group, and suppose that $\rho\ge 4$. Then
$\s_\rho(G)\le\frac32\,\rho^{-1}|G|$, and equality holds if and only if $G$
has a subgroup $H$ such that $G/H$ is cyclic of order $2\rho$. Indeed, the
following assertions are equivalent:
\begin{itemize}
\item[(i)] $A$ is a generating set for $G$ satisfying
  $\diam_A(G)\ge\rho$ and $|A|=\frac32\,\rho^{-1}|G|$;
\item[(ii)] $A=(-g+H)\cup H\cup(g+H)$, where $H$ is a subgroup of $G$ such
  that $G/H$ is cyclic of order $2\rho$, and $g$ is an element of $G$ such
  that $g+H$ is a generator of $G/H$.
\end{itemize}
\end{theorem}

This theorem shows that the maximum possible cardinality of a
generating set $A$ with $\diam_A(G)\ge\rho$ is $\frac32\,\rho^{-1}|G|$
and it describes the structure of all $A$ with precisely this
cardinality. Our next result goes beyond this, establishing the
structure of those generating sets $A$ with $\diam_A(G)\ge\rho$ and
$|A^\pm|>\frac4{3\rho-1}\,|G|$.

\begin{theorem}\label{t:thm213}
Let $G$ be a finite abelian group, and suppose that $\rho\ge 4$. If $A$ is a
generating set for $G$ such that $\diam_A(G)\ge\rho$ and
$|A^\pm|>4|G|/(3\rho-1)$, then there exist $H\le G$ and $g\in G$ satisfying
the following conditions:
\begin{itemize}
\item[(i)]  $A\seq (-g+H)\cup H\cup(g+H)$, and
  $|A^\pm|>\big(3-\frac{\rho-3}{3\rho-1}\big)|H|$;
\item[(ii)] $G/H$ is cyclic of order $2\rho \le |G/H| \le
    \frac94\,\rho-1$, and $g+H$ generates $G/H$.
\end{itemize}
\end{theorem}


\section{Context and motivation}\label{s:context}

As far as we know, diameters of finite groups have never been studied
systematically. Nevertheless, many individual problems have been considered
and to a certain extent solved. We provide several selected examples, mainly
to illustrate the various interconnections with other areas of research.

A lot of attention has been given to homocyclic groups of exponent $2$, that
is groups isomorphic to $\Z_2^r$ for some $r\in\N_0$. There are at least
three good reasons for this.

\smallskip\noindent (i) \emph{Connections with the Covering Radius in
  Coding Theory.}\quad There is a natural correspondence between
generating sets $A$ for $\Z_2^r$ and linear binary block codes of
co-dimension $r$ and minimum distance $d \ge 3$. (Given such a set
$A$, one can arrange its non-zero elements in columns to obtain a
check matrix of the code associated to $A$. For more details, see
\cite[Section 18.1]{b:chll} or \cite{b:cz}.) The length of the code
corresponding to $A$ is $|A\stm\{0\}|$ and its covering radius is
$\diam_A(\Z_2^r)$. This shows that codes of minimum distance $d\ge 3$,
co-dimension $r$, and covering radius at least $\rho$ exist if and
only if $\rho\in[1,r]$; in this case the maximum possible length of
such a code is $\s_\rho(\Z_2^r)-1$.

\smallskip\noindent (ii) \emph{Connections with Caps in Projective
  Geometries.}\quad A cap in a projective geometry is a collection of
points no three of which are collinear. In a finite projective geometry ${\rm
PG}(r-1,2)$ over the field with two elements, points can be identified with
the non-zero elements of $\Z_2^r$, and then caps correspond to subsets
$A\seq\Z_2^r\stm\{0\}$ such that no three elements of $A$ add up to zero.
Evidently, if $A$ is a cap, then for any $a\in A$ we have $a\notin\<A+a\>_3$,
whence $\diam_{A+a}(\Z_2^r)\ge 4$; moreover, if $A$ is not contained in a
hyperplane of ${\rm PG}(r-1,2)$, then $A+a$ generates $\Z_2^r$. Conversely,
suppose that $B$ is a generating set for $\Z_2^r$ such that
$\diam_B(\Z_2^r)\ge 4$ and $0\in B$. Then for any $g\in \Z_2^r \stm \<B\>_3$
the set $A:=B+g$ is a cap, not contained in a hyperplane. Therefore the
maximum size of a cap, not contained in a hyperplane, is $\s_4(\Z_2^r)$. For
(much) more on caps we refer the reader to~\cite{b:bhl,b:fp}.

\smallskip\noindent
(iii) \emph{Connections with Sum-free Sets in Combinatorial Number Theory.}
\quad A subset $A \seq G$ of an abelian group $G$ is said to be sum-free if
no two elements of $A$ add up to another element of $A$. For $G = \Z_2^r$
this reduces to $a_1+a_2+a_3\neq 0$ for all $a_1,a_2,a_3 \in A$. As shown
above, sets with this property are translates of sets of diameter greater
than three. Thus the maximum size of a sum-free subset of $\Z_2^r$, not
contained in any coset of any proper subgroup, is $\s_4(\Z_2^r)$. More
information on sum-free subsets of $\Z_2^r$ can be found in~\cite{b:cp}.

\smallskip The remarks in (i) and (iii) lead to further connections
with coding theory: it is not difficult to see that the code $C$,
corresponding to a maximal (under inclusion) sum-free subset of
$\Z_2^r$, has minimum distance $d\in\{4,5\}$ and covering radius
$\rho=2$. If $d=5$, then $C$ is a perfect code; in fact, it is known
that the only code with $d=5$ and $\rho=2$ is the repetition code of
length $5$, which has co-dimension $r=4$. For $r>4$ we necessarily
have $d=4$, and $C$ is a quasi-perfect code. Thus $\s_4(\Z_2^r)$ can
be interpreted as the maximum length of a ``non-trivial''
quasi-perfect code of co-dimension $r$ and covering radius $\rho = 2$.
(``Trivial'' quasi-perfect codes correspond to sum-free sets that are
complements of index two subgroups. These codes are extensions of the
Hamming codes.)

In their remarkable paper \cite{b:dt}, Davydov and Tombak have shown that
\begin{equation}\label{e:equ31}
  \t_4(\Z_2^r)=2^{r-2}+1\quad\text{and}\quad \s_4(\Z_2^r)=5\cdot2^{r-4}
\end{equation}
for $r \geq 4$. We note that the first of these equalities is much
subtler than the second one. Indeed, the value of $\s_4(\Z_2^r)$ was
re-established independently by other authors \cite{b:bhl,b:cp},
whereas --- to our knowledge --- no alternative proof of the formula
for $\t_4(\Z_2^r)$ has been found. Using the first of equations
\eqref{e:equ31}, Davydov and Tombak were able to treat several related
problems; for instance, they found all possible lengths $n\ge
2^{r-1}+1$ of quasi-perfect codes of co-dimension $r$ and covering
radius $\rho=2$.

For $\rho \ge 5$, until recently only estimates and no precise
formulae for $\s_\rho(\Z_2^r)$ were known; see \cite{b:cz,b:z} or
\cite[Chapter~18]{b:chll}. The exact values are determined in a
forthcoming paper by one of the present authors \cite{b:lev}.
Concerning general bounds we mention that~\cite[Lemma~3]{b:hr} can be
regarded as a precursor of our Theorem~\ref{t:thm212}: in our
notation, it asserts that
$$
\diam_A(G) \le \max\lfr 2,\frac{3|G|}{2|A|}\rfr
$$
for every generating set $A$ of a finite abelian group $G$.

It is worth pointing out that there are also investigations of the diameters
of non-abelian finite simple groups. Apart from the 26 so-called sporadic
isomorphism classes, non-abelian finite simple groups are known to fall into
two categories: they are either alternating or of Lie type. In the
survey~\cite{b:bhkls} one finds a discussion of three types of generating
sets for these groups: ``worst'' (giving maximal diameter), ``average'' (a
random choice of a prescribed number of generators) and ``best'' (giving
minimal diameter while keeping the number of generators limited). Instead of
precise formulae the authors describe asymptotic bounds, as the group order
tends to infinity. One of the important, as yet unproven
conjectures~\cite[Section~2]{b:bhkls} states that for non-abelian finite
simple groups $G$, one has $\diam(G)=(\log |G|)^{O(1)}$ as $|G|\to\infty$.


\section{Auxiliary results}\label{s:auxiliary}

For later use we list three well-known results about abelian groups. The
first two are very basic lemmata, which will be used freely without further
reference.

\begin{lemma}
If $\{g_1 \longc g_r\}$ is a generating set for a finite abelian group $G$,
then $\exp(G) = \lcm(\ord(g_1)\longc\ord(g_r))$.
\end{lemma}

\begin{lemma}
If $G$ is a finite abelian group and $g \in G$ has order $\ord(g)=\exp(G)$,
then there exists a subgroup $H\le G$ such that $G = H \oplus \<g\>$.
\end{lemma}

The next result is a much deeper theorem due to Kneser.

\begin{theorem}[Kneser, \cite{b:kn1,b:kn2}; see also \cite{b:mann}]
 \label{t:kneser}
Let $A$ and $B$ be finite non-empty subsets of an abelian
group $G$ such that
$$
|A+B| \le |A|+|B|-1.
$$
Then, letting $H := \pi(A+B)$, we have
$$
|A+B| = |A+H|+|B+H|-|H|.
$$
\end{theorem}

Since, in the above notation, we have $|A+H|\ge |A|$ and $|B+H|\ge
|B|$, Theorem~\ref{t:kneser} shows that $|A+B|\ge |A|+|B|-|H|$. A
straightforward induction yields

\begin{corollary}\label{c:cor44}
Let $A_1\longc A_r$ be finite non-empty subsets of an abelian group $G$, and
write $H:=\pi(A_1\longp A_r)$. Then we have
$$
|A_1\longp A_r| \ge |A_1|\longp |A_r|-(r-1)|H|.
$$
In particular, if $A_1 \longp A_r$ is aperiodic, then
$$
|A_1\longp A_r|\ge|A_1|\longp|A_r|-(r-1).
$$
\end{corollary}

In fact, Theorem~\ref{t:kneser} and Corollary~\ref{c:cor44} are equivalent,
as the former can easily be derived from the latter. For this reason we often
refer to Corollary~\ref{c:cor44} simply as ``Kneser's theorem''.


\section{The absolute diameter}\label{s:abs_diam}

In this section we establish Theorem~\ref{t:thm21},
Theorem~\ref{t:thm22}, and Corollary~\ref{c:cor23} after proving some
subsidiary results.

\begin{lemma}\label{l:lem51}
  Let $G = G_1 \longop G_r$ be a finite abelian group, and let $A =
  A_1\longu A_r$ where $A_i \seq G_i$ for all $i \in [1,r]$. Then
  $$ \diam_{A}(G) = \diam_{A_1}(G_1) \longp \diam_{A_r}(G_r). $$
\end{lemma}

\begin{proof}
Given $g = g_1\longp g_r \in G$, with $g_i\in G_i$ for all $i\in[1,r]$, we
have
$$
l_{A}(g) = l_{A_1}(g_1)\longp l_{A_r}(g_r),
$$
and the assertion follows.
\end{proof}

\begin{lemma}\label{l:lem52}
  Let $n_1\longc n_r\in \N$ where $r \geq 2$. Suppose that for every
  $i\in[1,r]$ we have
  $$
  \lcm(n_1\longc n_r) > \lcm(n_1\longc n_{i-1},n_{i+1}\longc n_r).
  $$
  Then
  $$
  \lcm(n_1\longc n_r) \ge 2^{r-1}\max\{n_i\colon i\in[1,r]\}.
  $$
  Consequently,
  $$ \lcm(n_1\longc n_r) > n_1\longp n_r
  $$
  and
  $$
  \Big\lfloor\frac{\lcm(n_1\longc n_r)}2\Big\rfloor >
  \lfl\frac{n_1}2\rfl \longp \lfl\frac{n_r}2\rfl.
  $$
\end{lemma}

\begin{proof}
  Without loss of generality we can assume that $\max\{n_i\colon
  i\in[1,r]\}=n_1$. Since $\lcm(n_1\longc n_i)$ is a proper divisor of
  $\lcm(n_1\longc n_{i+1})$ for every $i\in[1,r-1]$, we have
  $$
  \lcm(n_1\longc n_r) \ge 2 \lcm(n_1\longc n_{r-1}) \ge 4
  \lcm(n_1\longc n_{r-2}) \ge \ldots \ge 2^{r-1} n_1.
  $$
  Next, as $2^{r-1}\ge r$, we obtain $\lcm(n_1\longc n_r) >
  n_1\longp n_r$. For the last assertion we note that, if $\lcm(n_1
  \longc n_r)$ is odd, then so are all of $n_1\longc n_r$.
\end{proof}

\begin{proof}[Proof of Theorem~\ref{t:thm21}]
  Suppose that $G\cong\Z_{m_1}\longop\Z_{m_r}$ where $1\neq m_1\longd
  m_r$.  Considering standard generating sets we already observed that
  $\diam(G)\ge \lfloor m_1/2 \rfloor\longp\lfloor m_r/2 \rfloor$. It
  remains to establish the reverse inequality.

  Let $A$ be a generating set for $G$, and let $g\in G$. We use
  induction on $r$ to verify that
  \begin{equation}\label{e:equ51}
    l_A(g)\le\lfloor m_1/2 \rfloor \longp \lfloor m_r/2 \rfloor.
  \end{equation}

  For $r=0$ there is nothing to prove, and we assume that $r \ge 1$.
  Put $s:=|A|$ and write $A=\{a_1\longc a_s\}$. Renumbering the
  elements of $A$, if necessary, we find $t \in [1,s]$ such that
  \begin{itemize}
  \item[(i)] the group $\<a_t, a_{t+1}\longc a_s\>$ contains an element $a$ of
    order $\ord(a)=m_r$;
  \item[(ii)] for every $i \in [t,s]$ the group
    $\<a_t\longc a_{i-1},a_{i+1}\longc a_s\>$ does not contain any elements of
    order $m_r$.
  \end{itemize}
  For later use we note that these conditions are equivalent to
  \begin{itemize}
  \item[(I)] $\lcm(\ord(a_t)\longc\ord(a_s))=m_r$;
  \item[(II)] $\lcm(\ord(a_t)\longc\ord(a_{i-1}),
    \ord(a_{i+1})\longc\ord(a_s))<m_r$ for all $i\in[t,s]$.
  \end{itemize}
  Regarding the interpretation of conditions (ii) and (II) for $t=s$
  we notice that $\<\varnothing\>=\{0\}$ and $\lcm(\varnothing)=1$, in
  accordance with standard definitions.

  We find a subgroup $H\le G$ such that $G=H \oplus \<a\>$ and
  $H\cong\Z_{m_1}\longop\Z_{m_{r-1}}$. Let $\bet_1\longc\bet_s\in\Z$
  such that $b_i:=a_i+\bet_i a\in H$ for all $i \in [1,s]$.

  We note that $H =\<b_1\longc b_s\>$. Indeed, given $h\in H$ we find
  $\lam_1\longc\lam_s\in \Z$ such that
  $$
  h = \sum_{i=1}^s \lam_i a_i = \sum_{i=1}^s \lam_i (b_i - \bet_i
  a) = \sum_{i=1}^s \lam_i b_i -\Big( \sum_{i=1}^s \lam_i \bet_i
  \Big) a, $$
  and $G=H\oplus\<a\>$ implies that $h=\sum_{i=1}^s \lam_i
  b_i \in \<b_1\longc b_s\>$.

  Now we are ready to prove~\eqref{e:equ51}. We find $h\in H$ and
  $\nu\in\Z$ such that $g=h+\nu a$. By the induction hypothesis, there
  are $\lam_1\longc\lam_s\in\Z$ such that $h=\sum_{i=1}^s \lam_i b_i$
  and $\sum_{i=1}^s |\lam_i|\le\sum_{i=1}^{r-1} \lfloor m_i/2
  \rfloor$. By (i), there exist $\mu_t\longc\mu_s \in \Z$ with
  $|\mu_i|\le\lfloor\ord(a_i)/2\rfloor$ for $i\in[t,s]$ so that
  $$ g = \sum_{i=1}^s \lam_i b_i + \nu a
       = \sum_{i=1}^s \lam_i a_i
          +\Big(\nu + \sum_{i=1}^s \lam_i \bet_i \Big) a
       = \sum_{i=1}^{t-1} \lam_i a_i + \sum_{i=t}^s \mu_i a_i.
  $$

  Recalling (I) and (II), we apply Lemma~\ref{l:lem52} to find
  $$
  l_A(g) \le \sum_{i=1}^{t-1} |\lam_i| + \sum_{i=t}^s |\mu_i| \le
  \sum_{i=1}^{r-1} \lfloor m_i/2 \rfloor + \sum_{i=t}^s \lfloor
  \ord(a_i)/2 \rfloor \le\sum_{i=1}^r \lfloor m_i/2 \rfloor
  $$
  as required.
\end{proof}

\begin{lemma}\label{l:lem53}
  Let $G$ be a finite abelian group, and suppose that $A\seq G$ is
  minimal (under inclusion) subject to $\diam_A(G) = \diam(G)$. Then
  $A$ is a standard generating set.
\end{lemma}

\begin{proof}
  Suppose that $G$ has type $(m_1,\ldots,m_r)$, that is
  $G\cong\Z_{m_1}\longop\Z_{m_r}$ where $1 \neq m_1\longd m_r$.  We
  use induction on $r$.

  The case $r=0$ is trivial and we assume that $r\ge 1$.
  Revisit the proof of Theorem~\ref{t:thm21}; write $A=\{a_1\longc
  a_s\}$ and define $t$ as before. Following the original argument
  (while keeping in mind that the last inequality in
  Lemma~\ref{l:lem52} is strict), the equation $\diam_A(G) = \diam(G)
  = \sum_{i=1}^r \lfloor m_i/2 \rfloor$ now shows that $t=s$, so
  $\ord(a_s)=m_r$.

  As in the proof of Theorem~\ref{t:thm21}, we find $H\le G$ such that
  $G=H\oplus\<a_s\>$, and consequently $H \cong
  \Z_{m_1}\longop\Z_{m_{r-1}}$.  Furthermore, for every $i\in [1,s-1]$
  we find $\bet_i \in \Z$ such that $b_i:=a_i+\bet_i a_s\in H$. By
  minimality of $A$, we have $b_i\neq 0$, and $b_i\neq b_j$ (unless
  $i=j$) for all $i,j\in [1,s-1]$.

  It is easily seen that $B:=\{b_1\longc b_{s-1}\}$ satisfies
  $\diam_B(H)=\sum_{i=1}^{r-1} \lfloor m_i/2 \rfloor = \diam(H)$ and
  is minimal subject to this condition. By the induction hypothesis,
  $B$ is a standard generating set for $H$. This yields $s=r$,
  $$
  G = H \oplus\<a_r\> = \<b_1\> \longop \<b_{r-1}\> \oplus \<a_r\>
  $$
  and without loss of generality we may assume that $\ord(b_j)= m_j$ for
  all $j\in [1,r-1]$ .

  To show that $A$ is a standard generating set for $G$, it is enough
  to prove that $\ord(a_j)\mid m_j$, or equivalently
  \begin{equation}\label{e:equ52}
    m_r\mid m_j \bet_j
  \end{equation}
  for all $j \in [1,r-1]$.

  For a contradiction, suppose that $j \in [1,r-1]$ is an index for
  which \eqref{e:equ52} fails. Fix $g\in G$ such that
  $l_A(g)=\diam(G)$ and write $g=\lam_1b_1\longp\lam_{r-1}b_{r-1}+\lam
  a_r$, with $-m_i / 2 < \lam_i \le m_i/2$ for $i \in [1,r-1]$.
  Substituting $b_i = a_i + \bet_ia_r$ we get $g =
  \lam_1a_1\longp\lam_{r-1}a_{r-1}+\lam_r a_r$ with $-m_r/2 < \lam_r
  \le m_r/2$. As $l_A(g)=\sum_{i=1}^r \lfloor m_i/2\rfloor$, we
  actually have $|\lam_i|=\lfloor m_i/2\rfloor$ for all $i\in [1,r]$.
  Define
  \begin{align*}
    \eps &:= \begin{cases}
               0 & \text{if $\lam_j = m_j/2$,} \\
               1 & \text{if $\lam_j = (m_j-1)/2$,} \\
              -1 & \text{if $\lam_j = -(m_j-1)/2$.}
           \end{cases}
  \intertext{and}
    \mu_i &:= \begin{cases}
      \lam_i      & \text{if $i\in[1,r-1],\,i\neq j$,} \\
      -\lam_j-\eps & \text{if $i=j$.}
    \end{cases}
  \end{align*}
  Furthermore, choose $\mu_r\in\Z$ such that $|\mu_r|\le\lfloor m_r/2
  \rfloor$ and $\mu_r\equiv\lam_r+(\mu_j-\lam_j)\bet_j\mmod{m_r}$.

  Notice that $|\mu_j|=|\lam_j|+|\eps|$ and $\mu_j\equiv
  \lam_j\mmod{m_j}$. The latter relation implies
  $$
  \mu_j a_j+\mu_r a_r = \mu_j(b_j-\bet_j a_r) + \lpr
  \lam_r+(\mu_j-\lam_j)\bet_j \rpr a_r = \lam_j a_j + \lam_r a_r,
  $$
  and hence $g=\mu_1 a_1\longp\mu_r a_r$.

  Moreover, $\mu_r\equiv\lam_r-m_j\bet_j\mmod{m_r}$ if $\eps=0$, and
  $\mu_r\equiv\lam_r-\eps m_j\bet_j\mmod{m_r}$ otherwise. From the
  fact that $m_j$ is a proper divisor of $m_r$ and $m_r\nmid m_j
  \bet_j$, it is not difficult to derive that $|\mu_r|\le
  |\lam_r|-m_j+1$. Thus $|\mu_j|+|\mu_r|\le
  |\lam_j|+|\eps|+|\lam_r|-m_j+1<|\lam_j|+|\lam_r|$.

  Therefore we obtain $g=\mu_1 a_1\longp\mu_r a_r$ where
  $\sum_{i=1}^r|\mu_i|<\sum_{i=1}^r |\lam_i|=l_A(g)$, a contradiction,
  as required.
\end{proof}

\begin{lemma}\label{l:lem54}
  \textup{(i)} Let $G=\Z_3\oplus \Z_3$ and put
  $A:=\{(1,0),(0,1),(1,1)\}^\pm$. Then $A$ satisfies
  $\diam_A(G)=\diam(G)$ and is maximal (under inclusion) subject to
  this condition. \\
  \textup{(ii)} Let $G=\Z_3\oplus\Z_3\oplus\Z_3$ and put
  $A:=\{(1,0,0),(0,1,0),(0,0,1),(1,\eps_1,0),(1,0,\eps_2)\}^\pm$,
  where $\eps_1,\eps_2\in\{-1,1\}$. Then $\diam_A(G) < \diam(G)$.
\end{lemma}

\begin{proof}
  (i) Theorem~\ref{t:thm21} shows that $\diam(G)=2$, and the claim
  follows easily from $A=G\stm\{(1,-1),(-1,1)\}$.

  \smallskip\noindent (ii) Theorem~\ref{t:thm21} shows that $\diam(G)
  = 3$. Applying an automorphism of $G$ if necessary, we may assume
  without loss of generality that $\eps_1=\eps_2=1$. Now it is easy to
  check that $\diam_A(G)=2$.
\end{proof}

\begin{proof}[Proof of Theorem~\ref{t:thm22}]
  Let $A \seq G$. If $\< A \> \neq G$, then neither (i) nor (ii)
  holds.  Now suppose that $A$ generates $G$ and show that assertions
  (i) and (ii) of the theorem are equivalent. One direction is easy:
  if (ii) holds, then Lemmata~\ref{l:lem51} and~\ref{l:lem54} show
  that $\diam_A(G)=\diam_B(G)$, as wanted.

  Now suppose that (i) holds, that is $\diam_A(G)=\diam(G)$. We can
  assume that $A^\pm = A$. By Lemma~\ref{l:lem53}, there is a standard
  generating set $B \seq A$.

  To simplify the notation, we assume further that $G= \Z_{m_1}
  \longop \Z_{m_r}$ where $1 < m_1 \mid \dotsb \mid m_r$ and that
  $B=\{(1,0,\ldots,0)\longc(0,\ldots,0,1) \}$. Let $g\in G$ with
  $l_A(g)=\diam(G)=\sum_{i=1}^r \lfloor m_i/2 \rfloor$. Then certainly
  $l_B(g)=\sum_{i=1}^r \lfloor m_i/2 \rfloor$, and so
  $g=(\lam_1,\ldots,\lam_r)$ for suitable $\lam_1\longc\lam_r\in\Z$
  with $|\lam_i|=\lfl m_i/2 \rfl$ for all $i\in[1,r]$.

  Now suppose that $a = (\alp_1 \longc \alp_r) \in A \stm B^\pm$. Then
  $a$ has at least two non-zero components. If it had more than two
  non-zero components, then $l_B (g+a)$ or $l_B (g-a)$ would be less
  than $l_B (g) - 1$, a contradiction to $l_A (g) = l_B (g)$. So
  exactly two components of $a$ are non-zero, say $\alp_i$ and
  $\alp_j$ where $i < j$.

  Moreover, we must have $\alp_i, \alp_j \in \{-1,1\}$, because
  otherwise $l_B(g+a)$ or $l_B(g-a)$ would again be less than
  $l_B(g)-1$. Our next aim is to show that $m_i=m_j=3$.

  If one of $m_i,m_j$ were even, then $l_B(g+a)$ or $l_B(g-a)$ would
  be less than $l_B(g)-1$, a contradiction. If $m_i$ and $m_j$ were
  both greater than $4$, then $l_B(g+2a)$ or $l_B(g-2a)$ would be less
  than $l_B(g)-2$, impossible. If $m_i=3$ and $m_j>3$, then either one
  of $l_B(g+a)$, $l_B(g-a)$ would be less than $l_B(g)-1$ or one of
  $l_B(g+2a)$, $l_B(g-2a)$ would be less than $l_B(g)-2$, again a
  contradiction.

  The only remaining possibility is $m_i = m_j = 3$, and in view of
  Lemma~\ref{l:lem51}, we are reduced to the case where $G$ is a
  homocyclic group of exponent $3$. Now (ii) follows from
  Lemma~\ref{l:lem54}.
\end{proof}

\begin{proof}[Proof of Corollary~\ref{c:cor23}]
  We have to bound the size of $A \subseteq G$ satisfying
  $\diam_A(G)=\diam(G)$. For brevity, we write $\nu_2:=\nu_2(G),\
  \nu_3:=\nu_3(G)$, and $r:=\rk(G)$.  Since $A$ generates $G$, we
  certainly have $r\le |A|$, and it remains to show that $|A| \le 1 +
  2 r - \nu_2 + 2 \lfloor \nu_3 /2 \rfloor$.

  Based upon Theorem~\ref{t:thm22}, we pick a standard generating set
  $B = \{b_1\longc b_r\}$ for $G$ such that
  $$
  A \seq \big(B \cup \{ b_{2i-1}+b_{2i} \colon i \in [1,\nu_3 /2]
  \}\big)^\pm.
  $$

  First suppose that $\nu_3 = 0$. Then we have $A \subseteq B^\pm$,
  and $B^\pm$ is the disjoint union of $\{0\}$, $\{ b_i \colon i \in
  [1,\nu_2] \}$, $\{ b_i \colon i \in [\nu_2 + 1,r] \}$, and $\{ -b_i
  \colon i \in [\nu_2 +1,r] \}$. This gives $|A| \leq 1 + \nu_2 + 2(r
  - \nu_2) = 1 + 2 r - \nu_2$, as wanted.

  Now suppose that $\nu_3 \neq 0$. Then $\nu_2 = 0$, and the set $(B
  \cup \{ b_{2i-1}+b_{2i} \colon i \in [1,\nu_3 /2] \})^\pm$ is the
  disjoint union of $\{0\}$, $B$, $-B$, $\{ b_{2i-1}+b_{2i} \colon i
  \in [1,\nu_3 /2] \}$, and $\{ -(b_{2i-1}+b_{2i}) \colon i \in
  [1,\nu_3 /2] \}$. This gives $|A| \leq 1 + 2r + 2 \lfloor \nu_3 /2
  \rfloor$, completing the proof.
\end{proof}


\section{The size of small diameter sets, I: First results}\label{s:bounds_I}

Let $G$ be a finite abelian group, and suppose that
$\rho\in[2,\diam(G)]$. In this section we collect some general
observations regarding the invariants $\t_\rho(G)$ and $\s_\rho(G)$,
and we determine them completely for $\rho=2$.  The assertions of
Lemma~\ref{l:lem24}, Example~\ref{p:exam25},
Proposition~\ref{p:prop26}, Corollary~\ref{c:cor28}, and
Proposition~\ref{p:prop211} will be proved.

We start with Lemma~\ref{l:lem24}, explaining the relation between
$\t_\rho(G)$ and $\s_\rho(G)$.

\begin{proof}[Proof of Lemma~\ref{l:lem24}]
  Recall that $\rho \in [2,\diam(G)]$. We have to show that
  $$
  \s_\rho(G) = \max \{ |H| \cdot \t_\rho(G/H) \colon H \lneqq G \}.
  $$

  Suppose that $A$ is a $\rho$-maximal generating set for $G$, and
  write $H := \pi(A) \lneqq G$. Then the image $\oA$ of $A$ under the
  canonical homomorphism $G \to G/H$ is $\rho$-maximal in $G/H$,
  non-zero, and aperiodic; see the discussion in
  Section~\ref{s:subsec22}. We get $\t_\rho(G/H)\ge |\oA| = |A|/|H|$,
  and if $A$ is chosen so that $|A|=\s_\rho(G)$, this yields
  $$
  \s_\rho(G) \le |H| \cdot \t_\rho(G/H).
  $$

  Conversely, let $H \lneqq G$ and suppose that $\oA$ is an aperiodic
  $\rho$-maximal generating set for $G/H$. Let $A$ denote the full
  pre-image of $\oA$ in $G$ under the canonical homomorphism $G \to
  G/H$. Then $A$ is a generating set for $G$ with $\diam_A(G) \geq
  \rho$. If $\oA$ is chosen so that $|\oA|=\t_\rho(G/H)$, we get
  $$
  \s_\rho(G) \ge |A| = |H| \cdot |\oA| = |H| \cdot \t_\rho(G/H).
  $$
\end{proof}

In connection with the equation $\t_1(G) = 0$ it was indicated that
$\t_\rho(G)$ may also vanish for certain $\rho\in [2,\diam(G)]$.
Example~\ref{p:exam25} describes a situation of this kind.

\begin{proof}[Explanation of Example~\ref{p:exam25}]
  Let $G = \Z_2 \oplus \Z_{2^{n+1}}$ with $n \geq 2$.
  Theorem~\ref{t:thm21} shows that $\diam(G) = 1 + 2^n$. Let $A \seq G$
  be maximal subject to $\<A\>_{2^n -1} \neq \<A\> = G$; we have to
  show that $A$ is periodic.

  By maximality, we have $A = A^\pm$. Since $A$ generates $G$, there
  exists an element $a_1 \in A$ of order $\ord(a_1) = 2^{n+1}$, and
  furthermore there exists an element $a_2 \in A$ such that $a_2
  \notin \<a_1\>$. Without loss of generality, we may assume that $a_1
  = (0,1)$ and $a_2 = (1,\alp)$ with $0 \leq \alp \leq 2^n$.

  Notice that, whenever $(1,\bet) \in A$ (this holds for instance for
  $\bet = \alp$), then
  \begin{equation}\label{e:equ61}
  \begin{split}
    G \stm \<A\>_{2^n -1} & \seq (\Z_2 \oplus \Z_{2^{n+1}}) \stm \<
    (1,\beta),(0,1) \>_{2^n -1} \\
    & \seq \{(0,2^n), (1,2^n-\bet), (1,2^n-1-\bet),
    (1,2^n+1-\bet)\}^\pm,
  \end{split}
  \end{equation}
  and because $\<A\>_{2^n -1} \neq G$, at least one of
  $(0,2^n)$, $(1,2^n-\bet)$, $(1,2^n-1-\bet)$, $(1,2^n+1-\bet)$
  has length greater than $2^n -1$ with respect to $A$.

  \noindent
  \emph{Assertion. If $b=(1,\bet)\in A$, then
    $b\in\{(1,0),(1,1),(1,2^n -1),(1,2^n)\}^\pm$.} \\
  For a contradiction, suppose that $b = (1,\bet) \in
  A\stm\{(1,0),(1,1),(1,2^n-1),(1,2^n)\}^\pm$. Then we have $2b = (0,2
  \bet) \notin \{(0,0),(0,1),(0,2),(0,3)\}^\pm$, and thus
  \begin{align*}
    &(0,2^n) - 2b = (0,2^n -2\bet) \in \<a_1\>_{2^n-4}, & & \text{so
      $l_A(0,2^n) \leq 2^n -2$;}\\
    &(1,2^n - \bet) - b = (0,2^n - 2\bet) \in
      \<a_1\>_{2^n-4}, & & \text{so $l_A(1,2^n- \bet) \leq 2^n - 3$;}\\
    &(1,2^n - 1 -\bet) - b = (0,2^n -1 -2\bet) \in
      \<a_1\>_{2^n-3}, & & \text{so $l_A(1,2^n-1- \bet) \leq 2^n - 2$;}\\
    &(1,2^n + 1 -\bet) - b = (0,2^n +1 -2\bet) \in
      \<a_1\>_{2^n-3}, & & \text{so $l_A(1,2^n+1- \bet) \leq 2^n - 2$.}
  \end{align*}
  This contradicts the observation following \eqref{e:equ61}.

  \noindent
  \emph{Assertion. If $c=(0,\gamma)\in A$, then $c\in\{(0,1)\}^\pm$.} \\
  Let $c = (0,\gamma) \in G \stm \{(0,1)\}^\pm$. Then it is easily
  seen that $\< a_1, c \>_{2^n-2} = \{ (0,\delta) \colon \delta \in
  \Z_{2^{n+1}} \}$, and hence $\< a_1, a_2, c \>_{2^n -1} = G$. It
  follows that $c \not \in A$, as required.

  The two assertions above yield
  $$
  \{ (0,1), (1,\alp) \}^\pm \seq A \seq \{ (0,1),
  (1,0), (1,1), (1,2^n -1), (1,2^n) \}^\pm.
  $$

  \noindent
  \emph{Case 1: $(1,\alp) = (1,0)$.} Because of \eqref{e:equ61} we
  certainly have $(1,2^n -1), (1,2^n) \notin A$. On the other hand,
  $(0,2^n) \notin \< (0,1),(1,0),(1,1) \>_{2^n - 1}$. So $A = \{
  (0,1), (1,0), (1,1) \}^\pm$ has non-trivial period $\{ (0,0), (1,0) \}$.

  \noindent
  \emph{Case 2: $(1,\alp) = (1,1)$.} Again \eqref{e:equ61} shows that
  $(1,2^n -1), (1,2^n) \notin A$. So as in in the first case $A = \{
  (0,1), (1,0), (1,1) \}^\pm$ has non-trivial period $\{ (0,0), (1,0)
  \}$.

  \noindent
  \emph{Case 3: $(1,\alp) \in \{ (1,2^n -1), (1,2^n) \}$.} Write the
  elements of $G$ with respect to the generating pair
  $((1,2^n),(0,1))$ rather than $((1,0),(0,1))$. In these new
  coordinates, $a_1$ is still represented by $(0,1)$, but $a_2$ is
  represented by $(1,0)$ or $(1,-1)$. We are reduced to either Case~1
  or to Case~2.
\end{proof}

We now determine $\t_2(G)$ and $\s_2(G)$. Observe that $\diam(G)\ge 2$
if and only if $|G|>3$ by Theorem~\ref{t:thm21}.

\begin{proof}[Proof of Proposition~\ref{p:prop26}]
  Let $A \seq G$. Then $A$ is $2$-maximal in $G$ if and only if it has
  the form $A = G\stm\{a,-a\}$ for some $a\in G\stm\{0\}$. Furthermore
  we have $\pi(A)=\pi(G\stm A)$, hence $A$ is $2$-maximal and
  aperiodic if and only if $A=G\stm\{a,-a\}$ for some $a\in G$ with
  $\ord(a)\notin\{1,4\}$.

  Now, if $|G|$ is even, then $|G\stm\{a,-a\}|=|G|-1$ for every $a\in
  G$ with $\ord(a)=2$. If $|G|$ is odd, then $|G\stm\{a,-a\}|=|G|-2$
  for every $a\in G\stm\{0\}$.
\end{proof}

Next we determine $\t_\rho(G)$ and $\s_\rho(G)$ for $\rho = \diam(G)$.

\begin{proof}[Proof of Corollary~\ref{c:cor28}]
  Recall that $G$ is a finite abelian group and $\rho=\diam(G)\ge 2$.
  Write $r:=\rk(G)$. Then Corollary~\ref{c:cor23}
  shows
  that $\s_\rho(G) = 1+2r-\nu_2(G)+2\lfl\nu_3(G)/2\rfl$, and it
  remains to prove that $\t_\rho(G)$ has the same value.

  In fact, Theorem~\ref{t:thm22} explains how to construct a
  $\rho$-maximal generating set for $G$: fix a standard generating set
  $B=\{b_1\longc b_r\}$, and put
  $$
  A := \big( B \cup \{b_{2i-1}+b_{2i} \colon i \in [1,\nu_3(G)/2]
  \} \big)^\pm.
  $$
  To conclude the proof we show that $A$ is aperiodic.

  For $r=1$ the claim is clear. Now consider the case $r\ge 2$. First
  suppose that $G$ is not homocyclic of exponent $3$, or that $r$ is
  odd. Then we have $(A+b_r) \cap A \seq \{0, b_r, -b_r\}$.  So
  $\{-b_1,-b_r,0\}\seq A$ implies that
  $$
  \pi(A)\seq (A+b_1)\cap(A+b_r)\cap A=\{0\}.
  $$
  Now suppose that $G$ is a homocyclic group of exponent $3$ and
  that $r$ is even. Then by construction $|A|=1+2r+r\equiv 1\mmod 3$.
  As $A$ is a union of $\pi(A)$-cosets, it follows that
  $\pi(A)=\{0\}$.
\end{proof}

We end this section with two more lemmata and a proof of
Proposition~\ref{p:prop211}.

\begin{lemma}\label{l:lem61}
  Let $G$ be a finite abelian group of odd order, and suppose that
  $\rho\in[2,\diam(G)]$. Then $\s_\rho(G)$ is odd, and $\t_\rho(G)$ is
  either zero or odd.
\end{lemma}

\begin{proof}
  Let $A$ be a $\rho$-maximal subset of $G$. Then $A=A^\pm$, and since
  $G$ contains no elements of order two, $|A|=|A^\pm|$ is odd.
\end{proof}

\begin{lemma}\label{l:lem62}
  Let $A$ be a $\rho$-maximal subset of a finite abelian group $G$,
  where $\rho\in[2,\diam(G)]$. Then $\pi(A)=\pi(\<A\>_\tau)$ for all
  $\tau\in[1,\rho-1]$.
\end{lemma}

\begin{proof}
  Clearly, we have $\pi(A)\seq\pi(\<A\>_2 )\longs\pi(\<A\>_{\rho-1})$.
  Write $H:=\pi(\< A\>_{\rho-1} )$. Since
  $$
  \<A+H\>_{\rho-1} = \<A\>_{\rho-1}+H = \<A\>_{\rho-1} \snqq G,
  $$
  by maximality of $A$ we have $A + H = A$, whence $H\seq\pi(A)$.
\end{proof}

\begin{proof}[Proof of Proposition~\ref{p:prop211}]
  Let $A$ be a $\rho$-maximal aperiodic subset of $G$. By
  Lemma~\ref{l:lem62} we have $\pi(\<A\>_\tau)=\{0\}$ for all
  $\tau\in[1,\rho-1]$. So Kneser's theorem (Corollary \ref{c:cor44}) shows
  that
  $$
  |G|-1 \ge |\<A\>_{\rho -1}| \ge (\rho-1)|A|-(\rho-2),
  $$
  hence $|A|\le (|G|-2)/(\rho-1)+1$, proving the first assertion.

  Note that the largest odd integer, not exceeding the right-hand side
  of this inequality,
  is $2\lfloor (|G|-2)/(2(\rho-1))\rfloor+1$.  So, if
  $\rk_2(G)=0$, the second assertion follows from Lemma~\ref{l:lem61}.

  Finally, suppose that $\rk_2(G)=1$, and let $g$ denote the unique
  element of order two in $G$. Put
  $k:=\lfloor(|G|-2)/(2(\rho-1))\rfloor$. By the above we have
  $$
  |A|\le \frac{|G|-2}{\rho-1} + 1 < 2k+3,
  $$
  and we need to show that in fact $|A|\le 2k+1$.

  For a contradiction, suppose that $|A|=2k+2$. Since $0\in A=A^\pm$
  and since $|A|$ is even, we conclude that $g \in A$ and therefore $g
  \in \<A\>_\tau$ for all $\tau \in [1,\rho-1]$. So for these values
  of $\tau$ the cardinalities $|\<A\>_\tau|$ are even, too. By
  Lemma~\ref{l:lem62}, the sets $\<A\>_\tau$ are aperiodic, and
  Kneser's theorem implies
  $$
  |\< A \>_\tau| \ge |\< A \>_{\tau-1}|+|A|-1\qquad (\tau \in[1,\rho-1]).
  $$
  Comparing the parities of the two sides (for each of these
  inequalities) we get
  $$
  |\< A \>_{\rho -1}| \ge |\<A\>_{\rho-2}|+|A| \longge (\rho-1)|A|
  = 2(\rho-1)(k+1) > |G|-2.
  $$
  Since $|\<A\>_{\rho -1}|$ and $|G|$ are both even, we have
  $\<A\>_{\rho-1}=G$, a contradiction.
\end{proof}


\section{The size of small diameter sets, II: Cyclic groups}
\label{s:bounds_II}

Let $G\cong\Z_m$ be the cyclic group of order $m$. According to
Theorem~\ref{t:thm21} we have $\diam(G)=\lfl m/2\rfl$. In this section
we prove Theorem~\ref{t:thm29}, thus determining $\t_\rho(G)$ and
$\s_\rho(G)$ for all $\rho \in [1,m/2]$.

\begin{proof}[Proof of Theorem~\ref{t:thm29}]
It suffices to show that
  $$  \t_\rho(\Z_m) = 2 \lfl \frac{m-2}{2(\rho-1)} \rfl + 1; $$
the formula for $\s_\rho(\Z_m)$ then follows from Lemma~\ref{l:lem24}.

Set $k := \lfloor (m-2)/(2(\rho-1))\rfloor$, so that $k\ge 1$ in view of
$\rho\le m/2$ and
\begin{equation}\label{e:Brho-1}
   2(\rho-1)k+1 \le (m-2)+1 < m.
\end{equation}
By Proposition~\ref{p:prop211}, we have $\t_\rho(\Z_m)\le 2k+1$. Therefore,
it suffices to exhibit a $\rho$-maximal aperiodic subset $A\seq\Z_m$ of
cardinality $|A|\ge 2k+1$.

Put $B:=\{-k,-k+1\longc k\}$. Then \eqref{e:Brho-1} yields
\begin{equation}\label{e:Brho-2}
  \<B\>_{\rho-1}=\{-(\rho-1)k,-(\rho-1)k+1\longc(\rho-1)k \} \neq \Z_m,
\end{equation}
and we pick a $\rho$-maximal subset $A \seq \Z_m$ containing $B$. It
is enough to show that $A$ is aperiodic; in fact this will imply that
$A=B$.

Writing $H:=\pi(A)$, we observe that
  $$ \<B\>_{\rho-1}+H \seq \<A\>_{\rho-1}+H = \<A\>_{\rho-1} \snqq \Z_m $$
by Lemma \ref{l:lem62}.
On the other hand, \eqref{e:Brho-2} shows
that
  $$ |\<B\>_{\rho-1}| = 2(\rho-1)k+1 > 2(\rho-1)\,
                                        \frac{(m-2)}{4(\rho-1)}+1 = m/2. $$
Thus $\<B\>_{\rho-1}+H \snqq \Z_m$ implies $H=\{0\}$, as claimed.
\end{proof}

For later use we record separately the case $\rho=3$.

\begin{corollary}\label{c:cor71}
  For every $m \in \N$ with $m \ge 6$ we have
  $$
  \t_3(\Z_m) = 2\lfl \frac{m-2}{4} \rfl +1,
  $$
  and
  $$
  \s_3(\Z_m) = \begin{cases}
    \frac{m-1}{2}-1 & \text{if $m \equiv 1\mmod 4$,} \\
    \frac{m-1}{2}   & \text{if $m \equiv 3\mmod 4$,} \\
    \frac{m}{2}-1   & \text{if $m$ is a power of $2$,} \\
    \frac{m}{2}     & \text{otherwise.}
               \end{cases}
  $$
\end{corollary}

\begin{proof}
  The expression for $\t_3(\Z_m)$ comes directly from
  Theorem~\ref{t:thm29}, which also yields
  \begin{equation}\label{e:equ71}
    \s_3(\Z_m) = \frac m2 - \min \lfr \frac{2m}d\,\lfr \frac{d-2}4
    \rfr \colon d\mid m,\,d\ge 6  \rfr,
  \end{equation}
  where $\{x\}:=x-\lfloor x \rfloor$ denotes the fractional part of
  $x \in \R$.

  If $m=2^k n$ with $n,k\in\N$ and $n\ge 3$ odd, then the minimum on
  the right hand side of \eqref{e:equ71} is attained for $d=2n\equiv
  2\mmod 4$.

  If $m = 2^k$ with an integer $k\ge 3$, then $\{(d-2)/4\}=1/2$ for
  any $d\mid m,\,d\ge 6$; therefore the minimum is attained for $d=m$.

  If $m \equiv 3\mmod 4$, then $\{(d-2)/4\}\ge 1/4=\{(m-2)/4\}$;
  again, the minimum is attained for $d = m$.

Finally, suppose that $m\equiv 1\mmod 4$. If $d\mid m$ and
 $d\equiv 1\mmod 4$, then
 $$
 \frac{2m}d\, \lfr\frac{d-2}4\rfr \ge \frac{2m}m\,
 \lfr\frac{m-2}4\rfr = \frac 32.
 $$
 If $d\mid m$ and $d\equiv 3\mmod 4$, then $d\le m/3$, and so we
 have
 $$
 \frac{2m}d\, \lfr\frac{d-2}4\rfr \ge \frac{2m}{m/3}\, \frac 14 =
 \frac 32.
 $$
 The minimum, once again, is attained for $d=m$.
\end{proof}


\section{The size of small diameter sets, III: Sets of diameter $3$}
\label{s:bounds_III}

Let $G$ be a finite abelian group. In this section we determine
$\t_3(G)$ and $\s_3(G)$, proving Theorem~\ref{t:thm27}.

\begin{lemma}\label{l:lem81}
  Let $G$ be a finite abelian group and let $H \le G$. Then for every
  $\rho\in[1,\diam(G/H)]$ we have $\s_\rho(G) \ge
  |H|\cdot\s_\rho(G/H)$. In particular, $\diam(G)\ge\diam(G/H)$.
\end{lemma}

\begin{proof}
  Given $\rho \in [1,\diam(G/H)]$, we choose a generating subset
  $\oA\seq G/H$ such that $\diam_{\oA}(G/H)\ge\rho$ and
  $|\oA|=\s_\rho(G/H)$.  Let $A\seq G$ denote the full pre-image of
  $\oA$ under the canonical homomorphism $G\to G/H$. Then $A$ is a
  generating subset of $G$ with $\diam_A(G)\ge\rho$ and
  $|A|=|H|\cdot|\oA|=|H|\cdot\s_\rho(G/H)$; cf.\ the discussion in
  Section~\ref{s:subsec22}.
\end{proof}

\begin{lemma}\label{l:lem82}
  Let $G$ be a finite abelian group with $\diam(G)\ge 3$. Then
  $$
  \s_3(G)\le |G|/2.
  $$
  Moreover, if $|G|$ is odd and $n:=\exp(G)$, then
  $$ \s_3(G)\le \begin{cases}
        \frac{n-1}{2n}\,|G| - 1 &\text{if $n \equiv 1\mmod 4$,}\\
        \frac{n-1}{2n}\,|G|     &\text{if $n \equiv 3\mmod 4$.}
                \end{cases}
  $$
\end{lemma}

\begin{proof}
  The first assertion is immediate from the boxing principle: if $A
  \seq G$ satisfies $|A|>|G|/2$, then $\<A\>_2=2A^\pm=G$.

  Now suppose that $|G|$ is odd. Let $A$ be a generating subset with
  $\diam_A(G)\ge 3$. Fix $z\in G\stm \<A\>_2$ and write $d:=\ord(z)$.
  Since $z \notin A-A$, for every $g\in G$ at least one of the
  elements $g$ and $g+z$ does not belong to $A$. Thus every
  $\<z\>$-coset in $G$ contains at most $(d-1)/2$ elements of $A$, and
  therefore
  $$
  |A| \le \frac{d-1}2\, \frac{|G|}d \le \frac{n-1}{2n}\,|G|.
  $$
  The second assertion now follows from Lemma~\ref{l:lem61}.
\end{proof}

\begin{lemma}\label{l:lem83}
Let $G$ be a non-trivial finite abelian group.
\begin{itemize}
\item[(i)] Suppose that $\rk_2(G)=0$. Then there exists an aperiodic
  subset $A\seq G$ such that $|A|=(|G|-1)/2$ and $0\notin 2A$.
\item[(ii)] Suppose that $\rk_2(G)>\rk_2(2\ast G)$ and
  $G\not\cong\Z_2\oplus\Z_2$. Then $\t_3(G)\ge |G|/2$.
\item[(iii)] Suppose that $\rk_2(G)=\rk_2(2\ast G)\ge 1$ and
  $\exp(G)>4$. Then $\t_3(G)\ge |G|/2 -1$.
\end{itemize}
\end{lemma}

\begin{proof}
  (i) As $G$ contains no elements of order two, $G\stm\{0\}$ is the
  disjoint union of $2$-subsets of the form $\{g,-g\}\ (g\in
  G\stm\{0\})$. Suppose that $A\seq G$ contains exactly one element
  from each of these $2$-subsets. Then it is immediate that
  $|A|=(|G|-1)/2$ and $0\notin 2A$. Since $\gcd(|G|,|A|)=1$, the set
  $A$ is aperiodic; see Section~\ref{s:subsec22}.

  \smallskip\noindent (ii) Fix an element $h \in G$ with $\ord(h)=2$
  and a subgroup $K\le G$ such that $G=K\oplus\<h\>$. Put
  $A:=\{0\}\cup(K+h)\stm\{h\}$. As $|A|=|G|/2$, it suffices to show
  that $A$ is $3$-maximal and aperiodic. The former follows from the
  first part of Lemma~\ref{l:lem82} and $\<A\>_2=G\stm\{h\}$. This
  also yields
  $\pi(A)\seq\pi(\<A\>_2)=\pi(G\stm\<A\>_2)=\pi(\{h\})=\{0\}$.

  \smallskip\noindent (iii) Fix an element $h\in G$ of order
  $\ord(h)>4$ and a subgroup $K\le G$ of index $[G:K]=2$ so that
  $G=K+\<h\>$. Put $A:=\{0\}\cup(K+h)\stm\{h,-h\}$.  Evidently, we
  have $|A|=|G|/2-1$.  Since $|G|$ is divisible by four, we obtain
  $\gcd(|G|,|A|)=\gcd(2,|G|/2-1)=1$. Therefore $A$ is aperiodic;
  see Section~\ref{s:subsec22}.

  It remains to show that $A$ is $3$-maximal. For this we observe that
  $\<A\>_2=G\stm\{h,-h\}$. For a contradiction, suppose that there
  exists $g\in G\stm A$ satisfying $\<\{g\}\cup A\>_2\neq G$. Then
  $g\in K\stm\{0\}$, and $h-g,h+g\notin A$. This implies that
  $h-g=-h=h+g$, hence $4h=2g=0$, a contradiction.
\end{proof}

\begin{lemma}\label{l:lem84}
  Let $G$ be a non-cyclic finite abelian group of odd order with
  $\diam(G)\ge 3$. Writing $n:=\exp(G)$, we have
  $$ \t_3(G) = \begin{cases}
  \frac{n-1}{2n}\,|G| -1  &\text{if $n \equiv 1\mmod 4$,}\\
  \frac{n-1}{2n}\,|G|     &\text{if $n \equiv 3\mmod 4$.}
               \end{cases} $$
\end{lemma}

\begin{proof}
  In view of Lemma~\ref{l:lem82}, it suffices to construct an
  aperiodic generating set $A$ for $G$ of diameter $\diam_A(G)\ge 3$
  and appropriate size.

  Since $G$ is not cyclic, we may assume that $G = H \oplus \Z_n$
  where $\{0\} \neq H \le G$. Elements of $G$ will be written as pairs
  $(h,x)$ with $h \in H$ and $x \in \Z_n$.

  By Lemma~\ref{l:lem83}\,(i), there exists an aperiodic subset $B\seq
  H$ of cardinality $|B| = (|H|-1)/2$ such that $0\notin 2B$, and all
  the more so, $0 \notin B$. Put $k:=\lfl n/4\rfl$ so that $k\ge 1$,
  unless $n=3$ (in which case $G$ is homocyclic of exponent $3$).

  If $n = 4k+1$, we define
  \begin{multline*}
    A := ( H + \{ (0,-k+1),(0,-k+2)\longc(0,k-2),(0,k-1) \} ) \\
    \cup ( B+\{(0,k)\} ) \cup ( -B-\{(0,k)\} ).
  \end{multline*}
  If $n=4k+3\ge 7$, we define
  \begin{multline*}
    A := ( H + \{ (0,-k+1),(0,-k+2)\longc(0,k-2),(0,k-1) \} ) \\
    \cup ( B+\{(0,k),(0,k+1)\} ) \cup ( -B-\{(0,k),(0,k+1)\} ) \\
    \cup \{(0,k),(0,-k)\}.
  \end{multline*}
  Finally, if $n=3$, we define
  $$
  A := ( B+\{(0,1)\} ) \cup ( -B-\{(0,1)\} ) \cup \{(0,0)\}.
  $$

  We have $\pi(A) = 0$. Indeed, if $n \neq 3$, then projecting onto
  the second coordinate shows that $\pi(A) \seq H$, and thus
  $\pi(A) \seq \pi(B) = \{0\}$. If $n =3$, it is likewise easy to
  see that $\pi(A) \seq H$ and thus $\pi(A) \seq \pi(B) =
  \{0\}$.

Moreover, it is easily verified that $\<A\>=G$ and $\diam_A(G)\ge 3$: the
latter follows from
  $$ \<A \>_2 \not\ni \begin{cases}
          (0,2k)   &\text{if $n = 4k+1$,}\\
          (0,2k+1) &\text{if $n = 4k+3$.}
                      \end{cases}
  $$
  Finally, we count
  $$ |A| = \begin{cases}
       (2k-1)|H| + 2|B| = 2k|H| - 1     &\text{if $n = 4k+1$,}\\
       (2k-1)|H| + 4|B| + 2 = (2k+1)|H| &\text{if $n = 4k+3 \ge 7$,}\\
        2|B|+1 = |H|                    &\text{if $n=3$.}
           \end{cases}
  $$
  Noticing that $|H|=|G|/n$ we get
  $$ |A| = \begin{cases}
  \frac{n-1}{2n}\,|G| -1 &\text{if $n \equiv 1\mmod 4$,}\\
  \frac{n-1}{2n}\,|G|    &\text{if $n \equiv 3\mmod 4$.}
           \end{cases}
  $$
\end{proof}

\begin{lemma}\label{l:lem85}
  Let $G$ be a finite abelian group with $\diam(G)\ge 3$ and
  $\rk_2(G)>\rk_2(2\ast G)$. Then $\t_3(G)=\s_3(G)=|G|/2$.
\end{lemma}

\begin{proof}
  Lemma~\ref{l:lem82} shows that $\t_3(G)\le\s_3(G)\le |G|/2$, while by
  Lemma~\ref{l:lem83}\,(ii) we have $\t_3(G)\ge |G|/2$.
\end{proof}

\begin{lemma}\label{l:lem86}
  Let $G$ be a finite abelian group and let $A \seq G$. Suppose that
  $2\ast G\not\seq\<A\>_2$. Then
  $$
  |A| \le (|G|-2^{\rk_2(G)})/2.
  $$
\end{lemma}

\begin{proof}
  Fix $g\in G$ such that $2g\notin\<A\>_2$, and put $X:=\{x\in G\colon
  2x=0\}$. Then the sets $A-g$, $A+g$, and $X$ are pairwise disjoint,
  whence
  $$
  |A| \le (|G|-|X|)/2 = (|G|- 2^{\rk_2(G)})/2.
  $$
\end{proof}

\begin{lemma}\label{l:lem87}
  Let $G$ be a finite abelian group with $\rk_2(G)=\rk_2(2\ast G)\ge
  1$ and $\exp(G)>4$. Then we have $\t_3(G)=|G|/2-1$.
\end{lemma}

\begin{proof}
  In view of Lemmata~\ref{l:lem82} and \ref{l:lem83}(iii), it suffices to
  show that $\t_3(G)\neq |G|/2$. For a contradiction, suppose that $A$
  is a $3$-maximal aperiodic subset of $G$ of cardinality $|A|=|G|/2$.
  Kneser's theorem and Lemma \ref{l:lem62} show that
  $|\<A\>_2|=|G|-1$,
  so $\<A\>_2=G\stm\{h\}$ for some $h\in\{x\in G \colon 2x=0\} \seq
  2\ast G$. Now by Lemma~\ref{l:lem86} we have $|A|\le
  (|G|-2)/2=|G|/2-1$, a contradiction.
\end{proof}

\begin{lemma}\label{l:lem88}
  Let $G = \Z_4 \longop \Z_4$ be a homocyclic group
  of exponent $4$ and rank $r\ge 2$. Suppose that $A\seq G$ is a
  $3$-maximal subset such that $(2\ast G)\cup\<A\>_2\neq G$. Then $A$
  is periodic.
\end{lemma}

\begin{proof}
  Choose $h\in G\stm ((2\ast G)\cup\<A\>_2)$, so that in particular
  $\ord(h)=4$. We claim that $2h\in\pi(A)$. Indeed, it is enough to
  show that $a+2h\in A$ for any given $a\in A$, and since $A$ is
  $3$-maximal this will follow from $h\notin\<A\cup\{a+2h\}\>_2$.

  Since $h \notin 2 \ast G$, we have $h\neq 2 (a+2h)$. Since $a\in A$
  and $h \notin \<A\>_2$, we have $h-(a+2h) = -a-h \notin A$ and
  $h+(a+2h) = a-h \notin A$.  This shows that $h \notin \< A \cup
  \{a+2h\} \>_2$.
\end{proof}

\begin{lemma}\label{l:lem89}
Let $G=\Z_4\longop\Z_4$ be a homocyclic group of exponent $4$ and rank
 $r\ge 2$. Then $\t_3(G)=(|G|-\sqrt{|G|})/2=(4^r-2^r)/2$.
\end{lemma}

\begin{proof}
  From Lemmata~\ref{l:lem86} and~\ref{l:lem88} it follows that
  $\t_3(G)\le(|G|-2^{\rk_2(G)})/2=(4^r-2^r)/2$. Thus, it remains to
  construct a $3$-maximal aperiodic subset $A\seq G$ of size
  $|A|=(4^r-2^r)/2$. We use induction on $r$.

  If $r=2$, then $G=\Z_4\oplus\Z_4$, and we define
  $$
  A := \{(k,0) \colon k\in\Z_4 \} \cup \{(1,1),(-1,-1)\}.
  $$
  Here $|A|=6=(16-4)/2$, and it is readily checked that $A$ is a
  $3$-maximal aperiodic subset of $G$. Indeed, we have
  $\<A\>_2=G\stm\{(0,2),(1,2),(-1,2)\}$.

  Now suppose that $r\ge 3$. We write $G=H\oplus\Z_4$, where $H$ is
  a homocyclic group of exponent $4$ and rank $r-1$. By the induction
  hypothesis, there exists a $3$-maximal aperiodic subset $B\seq H$ with
  $|B|=(4^{r-1}-2^{r-1})/2$. Put
  $$
  A := \{(h,0) \colon h\in H \} \cup \{ (b,1) \colon b\in B \} \cup
  \{(-b,-1) \colon b\in B \}. $$
  Then
  $|A|=|H|+2|B|=4^{r-1}+4^{r-1}-2^{r-1}=(4^r-2^r)/2$, and it is
  readily checked that $A$ is $3$-maximal and aperiodic. Indeed, we
  have $\<A\>_2=G\stm\{(h,2)\colon h\notin\<B\>_2\}$.
\end{proof}

\begin{lemma}\label{l:lem810}
  Let $G$ be a finite abelian group. Suppose that $G$ is not cyclic
  and $\rk_2(G)=\rk_2(2\ast G)\ge 1$. Then $\s_3(G)=|G|/2$.
\end{lemma}

\begin{proof}
  First suppose that $\rk_2(G)=1$. Then we can write $G=H\oplus\Z_m$,
  where $m$ is even and not a power of $2$. By Lemmata~\ref{l:lem82}
  and~\ref{l:lem81} and Corollary~\ref{c:cor71} we have
  $$
  |G|/2 \ge \s_3(G) \ge |H|\cdot \s_3(\Z_m) = |G|/2,
  $$
  and the claim follows.

  Now suppose that $\rk_2(G)\ge 2$. In this case we find $H\le G$ such
  that $G/H\cong\Z_2\oplus\Z_4$. Considering the set
  $A=\{(0,0),(1,0),(0,1),(0,3)\}\seq\Z_2\oplus\Z_4$, one concludes
  easily that $\s_3(\Z_2\oplus\Z_4)\ge 4$, and the proof can be
  completed as above.
\end{proof}

\begin{proof}[Proof of Theorem~\ref{t:thm27}]
  (i) Suppose that $|G|$ is odd. If $G$ is cyclic, the assertion
  follows from Corollary~\ref{c:cor71}. If $G$ is not cyclic, the
  assertion follows from Lemmata~\ref{l:lem82} and~\ref{l:lem84}.

  \smallskip\noindent (ii) Suppose that $|G|$ is even. If $G$ is
  cyclic, the assertion follow from Corollary~\ref{c:cor71}. Now
  suppose that $G$ is not cyclic. Then the claim regarding $\t_3(G)$
  follows from Lemmata~\ref{l:lem85}, \ref{l:lem87}, and~\ref{l:lem89}.
  Finally, the claim regarding $\s_3(G)$ follows from
  Lemmata~\ref{l:lem85} and~\ref{l:lem810}.
\end{proof}


\section{The size of small diameter sets, IV: General estimates}
\label{s:bounds_IV}

In this section we prove Theorems~\ref{t:thm212} and~\ref{t:thm213}.

\begin{proof}[Proof of Theorem~\ref{t:thm212}]
  It suffices to show that $\s_\rho(G)\leq\frac32\,\rho^{-1}|G|$ and
  that (i) implies (ii). Indeed, (ii) trivially implies (i), and it is
  easily seen that, if (i) and (ii) are equivalent, then equality
  holds in $\s_\rho(G)\leq\frac32\,\rho^{-1}|G|$ if and only if $G$
  has a subgroup $H$ such that $G/H$ is cyclic of order $2\rho$.

  Let $A$ be a generating set for $G$ with $\diam_A(G)\ge\rho$. We
  write $H:=\pi(\<A\>_{\rho-1})$ and note that both $G$ and
  $\<A\>_{\rho-1}$ are unions of $H$-cosets, whence
  $|G|\ge|\<A\>_{\rho-1}|+|H|$. Since
  $$
  \<A\>_{\rho-1} = (\rho-1)A^\pm+H = (\rho-1)(A^\pm+H),
  $$
Kneser's theorem yields
  $$
  |\<A\>_{\rho-1}| \ge (\rho-1)|A^\pm+H| - (\rho-2)|H|.
  $$
  Combining these observations, we obtain
  \begin{equation}\label{e:equ91}
     |G| \ge (\rho-1)|A^\pm+H| - (\rho-3)|H|.
  \end{equation}

  Since $0\in A^\pm\not\seq H$, the set $A^\pm+H$ is the union of at
  least two $H$-cosets. Moreover, if $A^\pm$ were the union of
  \emph{exactly two} $H$-cosets, then $A^\pm+H=H\cup(g+H)$ for some $g
  \in G \stm H$ satisfying $g+H=-g+H$, and hence $2g\in H$. But this
  would yield
  $$
  G = \<A^\pm\> \seq \< A^\pm + H \> = H \cup (g+H) =
  (\rho-1)(A^\pm + H) = \<A\>_{\rho-1} \snqq G,
  $$
  a contradiction.

  We conclude that $A^\pm+H$ consists of at least three $H$-cosets,
  and so $|A^\pm+H|\ge 3|H|$. From \eqref{e:equ91} we obtain
  \begin{equation}\label{e:equ92}
    |G| \ge \Big(\rho-1-\frac13\, (\rho-3)\Big) |A^\pm+H| \ge
     \frac23\, \rho |A|.
  \end{equation}
  This gives the upper bound $\s_\rho(G) \leq \frac32 \, \rho^{-1}
  |G|$; it remains to show that (i) implies (ii).

  To this end, in addition to our earlier assumption that $A$ is a
  generating set for $G$ with $\diam_A(G)\ge\rho$, we assume that $|A|
  = \frac32\, \rho^{-1}|G|$. Then \eqref{e:equ92} yields $A=A^\pm+H$,
  and this set is the union of exactly three $H$-cosets. Consequently,
  there are two possibilities: either $A = H \cup(g_1+H) \cup(g_2+H)$,
  where $g_1+H,g_2+H \in G/H$ are distinct elements of order two, or
  $A=(-g+H) \cup H \cup (g+H)$, where $g+H \in G/H$ is of order
  greater than two. In the first case we would have
  $$
  G = \<A\> = \{0,g_1,g_2,g_1+g_2\}+H = \<A\>_{\rho-1} \snqq G,
  $$
  a contradiction. Thus we are left with the second case: $A=(-g+H)
  \cup H \cup (g+H)$, where $g+H \in G/H$ has order strictly greater
  than two.

  Since $A$ generates $G$, the quotient group $G/H$ is cyclic and
  generated by $g+H$. Since
  $$
  \bigcup\, \{ kg+H \colon -\rho+1 \leq k \leq \rho-1 \} =
  \<A\>_{\rho-1}\snqq G,
  $$
  the order of $G/H$ is at least $2\rho$. As $|A| = \frac32\,
  \rho^{-1}|G|$, the estimate
  $$
  |A| = 3|H| = 3\,\frac{|G|}{|G/H|} \le 3\,\frac{|G|}{2\rho} = |A|
  $$
  yields $|G/H| = 2\rho$, as desired.
\end{proof}

\begin{proof}[Proof of Theorem~\ref{t:thm213}]
  Let $A \seq G$ be a generating set such that $\diam_A(G)\ge\rho$ and
  $|A^\pm|>4|G|/(3\rho-1)$. We set $H:=\pi(\<A\>_{\rho-1})$ and show
  that there exists $g\in G$ such that
  \begin{itemize}
  \item[(i)] $A\seq (-g+H)\cup H\cup(g+H)$, and $|A^\pm| >
    \big(3-\frac{\rho-3}{3\rho-1}\big) |H|$;
  \item[(ii)] $G/H$ is cyclic of order $2\rho\le |G/H|\le
    \frac94\,\rho-1$, and $g+H$ generates $G/H$.
  \end{itemize}

  As in the proof of Theorem~\ref{t:thm212} we see that $A^\pm+H$
  consists of three $H$-cosets: the possibility $|A^\pm+H|\ge 4|H|$ is
  ruled out, since otherwise
  $$
  |G|\ge \Big( \rho-1-\frac14\, (\rho-3) \Big) |A^\pm+H| \ge
  \frac{3\rho-1}4\,|A^\pm|,
  $$
  contrary to the assumptions. Exactly as before, the quotient
  group $G/H$ is cyclic of order at least $2\rho$, and there exists
  $g\in G$ such that $g+H$ is a generator of $G/H$ and
  $A^\pm+H=(-g+H)\cup H\cup (g+H)$.

  Finally, from
  $$
  3|H|=|A^\pm+H|\ge|A^\pm|>4|H||G/H|/(3\rho-1)
  $$
  it follows that $|G/H|\le\frac94\rho-1$, and from
  $$
  |A^\pm| > |H|\,\frac{4|G/H|}{3\rho-1} \ge
  |H|\,\frac{8\rho}{3\rho-1}
  $$
  we obtain $|A^\pm|>(3-(\rho-3)/(3\rho-1))|H|$, as required.
\end{proof}

In fact one can go further and, given $\eps>0$, describe in a similar
way the structure of those generating sets $A$ for $G$ which satisfy
$\diam_A(G)\ge\rho$ and $|A^\pm|>(1+\eps)|G|/(\rho-1)$. Indeed, in
this case $A^\pm+H$ consists of at most $\lceil\eps^{-1}\rceil$
$H$-cosets. However, for small $\eps$ the description is likely to
become quite complicated.

\medskip

\noindent
\textbf{Acknowledgment.} We thank the referee for reading our
manuscript so carefully. His (or her) suggestions and comments led to
several improvements of exposition.


\end{document}